\documentclass{article}

\usepackage{graphicx}
\usepackage{amssymb,latexsym,amsthm,amsmath}
\usepackage{verbatim}
\usepackage{amsfonts}
\usepackage{algorithm}
\usepackage{latexsym}        
\usepackage[misc,geometry]{ifsym}

\theoremstyle{plain}
\newtheorem{theorem}{Theorem}[section]

\newtheorem{lemma}[theorem]{Lemma}
\newtheorem*{conjecture}{Conjecture}

\theoremstyle{definition}
\newtheorem{definition}{Definition}[section]
\newtheorem{example}{Example}[section]
\newtheorem{counter}{Counter-Example}[section]
\newtheorem{assumption}{Assumption}[section]

\theoremstyle{remark}

\numberwithin{equation}{section}

\newcommand{\R}{{\mathbb R}}

\begin{document}

\title{On a conjecture in second-order optimality conditions}

\author{Roger Behling\footnote{Federal University of Santa Catarina, Blumenau-SC, Brazil. Email: rogerbehling@gmail.com}\and
  Gabriel Haeser\footnote{Department of Applied Mathematics, University of S\~ao Paulo, S\~ao Paulo-SP, Brazil. This research was partially conducted while holding a Visiting Scholar position at Department of Management Science and Engineering, Stanford University, Stanford CA, USA. Email: ghaeser@ime.usp.br \Letter} \and Alberto Ramos\footnote{Department of Mathematics,
     Federal University of Paran\'a, Curitiba, PR, Brazil.
     e-mail: albertoramos@ufpr.br.}
    \and Daiana S. Viana\footnote{Federal University of Acre, Center of Exact and Technological Sciences, Rio Branco-AC, Brazil. PhD student at Department of Applied Mathematics, University of S\~ao Paulo-SP, Brazil. Email: daiana@ime.usp.br}
}

\date{June 16th, 2016. Last reviewed on June 23rd, 2017.}
 
\maketitle

\begin{abstract}
In this paper we deal with optimality conditions that can be verified by a nonlinear optimization algorithm, 
where only a single Lagrange multiplier is avaliable.
In particular, we deal with a conjecture formulated in 
[{\it R. Andreani, J.M. Mart\'{\i}nez, M.L. Schuverdt, ``On second-order optimality conditions for nonlinear programming'', Optimization, 56:529--542, 2007}], 
which states that whenever a local minimizer of a nonlinear optimization problem 
fulfills the Mangasarian-Fromovitz Constraint Qualification and the 
rank of the set of gradients of active constraints increases at most by one in a neighborhood of the minimizer, a second-order optimality condition that depends on one single Lagrange multiplier is satisfied. 
This conjecture generalizes previous results under a constant rank assumption or under a rank deficiency of at most one. 
In this paper we prove the conjecture under the additional 
assumption that the Jacobian matrix has a smooth singular value decomposition, which is weaker than previously considered assumptions. We also review previous literature related to the conjecture.
\end{abstract}

\noindent {\bf Keywords:}
Nonlinear optimization, Constraint qualifications, Second-order optimality
conditions, Singular value decomposition.

\noindent {\bf AMS Classification:} 90C46, 90C30

\pagestyle{myheadings}
\thispagestyle{plain}
\markboth{}{R. Behling, G. Haeser, A. Ramos, D. S. Viana, On a conjecture in second-order
optimality conditions}

\section{Introduction}
This paper considers a conjecture about second-order necessary optimality conditions for 
constrained optimization. Our interest in such conjecture comes from 
practical considerations. Numerical optimization 
deals with the design of algorithms with the aim of finding a point with the lowest possible value of a certain function over a constraint set.
Useful tools for the design of algorithms are the necessary optimality conditions, i.e., 
conditions satisfied by every local minimizer. 
Not all necessary optimality conditions serve  that purpose. 
Optimality conditions must be 
computable with the information provided by the algorithm, 
where its fulfillment indicates that the considered point is an acceptable solution.
For constrained optimization problems, 
the Karush-Kuhn-Tucker (KKT) conditions are the basis for most optimality conditions. 
In fact, most algorithms for constrained optimization are iterative and in their 
implementation,  
the KKT conditions serve as a theoretical guide for developing suitable 
stopping criteria. 
For more details, see \cite[Framework 7.13, page 513]{nocedal}, \cite[Chapter 12]{fletcher} and \cite{amrs}.

Necessary optimality conditions can be of first- or second-order 
depending on whether the first- or second-order derivatives are used in the formulation.
When the second-order information is avaliable, 
one can formulate second-order conditions. Such
conditions are much stronger than first-order ones and hence are mostly desirable, since they allow us to rule out possible non-minimizers accepted as solution when we only use first-order information. 

Global convergence proofs of second-order algorithms are based on second-order necessary optimality condition of the form: If a local
minimizer satisfies some constraint qualification, then the WSOC condition holds,
where WSOC stands for the Weak Second-order Optimality Condition, that states that the Hessian of
the Lagrangian at a KKT point, for some Lagrange multiplier, is positive semidefinite on the subspace orthogonal to the gradients of
active constraints, see Definition \ref{def:wsoc}. 

Thus, we are interested in assumptions guaranteeing that local minimizers satisfy WSOC, given its implications to numerical algorithms.
The conjecture comes along these lines. In order to precisely state the conjecture, we need some definitions. 

Consider the nonlinear constrained optimization problem
\begin{equation}
\label{problem}
            \begin{array}{lll}
           \mbox{minimize   } & f(x), \\ 
           \mbox{subject to } & h_{i}(x)  = 0 \ \ \forall i \in \mathcal{E}:=\{1,\dots, m\},\\
                              & g_{j}(x) \leq 0 \ \ \forall j \in \mathcal{I}:=\{1,\dots,p\}, \\
                              
            \end{array}
\end{equation}
where $f$, $h_{i}$, $g_{j}: \mathbb{R}^{n} \rightarrow \mathbb{R}$ are assumed to be, 
at least, twice continuously differentiable functions.

Denote by $\Omega$ the feasible set of \eqref{problem}.
For a point $x \in \Omega$, we define $A(x):=\{j \in \mathcal{I}: g_{j}(x)=0\}$ to denote the set of indices of active inequalities.
A feasible point $x^*$ satisfies the Mangasarian-Fromovitz Constraint Qualification (MFCQ) 
if $\{\nabla h_i(x^*):i\in\mathcal{E}\}$ 
is a linearly independent set and there is a direction 
$d\in\R^n$ such that $\nabla h_i(x^*)^\mathtt{T}d=0, i \in \mathcal{E}$ and $\nabla g_j(x^*)^\mathtt{T}d<0$, 
$j\in A(x^*)$. 
Define by $J(x)$ 
the matrix whose first $m$ rows are formed by $\nabla h_i(x)^\mathtt{T}$, $i\in\mathcal{E}$ and the
remaining rows by $\nabla g_j(x)^\mathtt{T}, j\in A(x^*)$.

In \cite{ams2}, with the aim of 
stating
a verifiable condition guaranteeing global convergence of a second-order
augmented Lagrangian method to a second-order stationary point, 
the authors proposed a new condition \cite[Section 3]{ams2} suitable for that purpose. 
Furthermore, based on \cite{baccaritrad} and their recently proposed condition,  
they stated the following conjecture, see \cite[Section 5]{ams2}:

\begin{conjecture}
Let $x^*$ be a local minimizer of (\ref{problem}). Assume that:
\begin{enumerate}
 \item MFCQ holds at $x^{*}$, 
 \item the 
rank of 
$\{\nabla h_i(x), \nabla g_j(x): i \in \mathcal{E}; j \in A(x^*)\}$
is at most $r+1$ in a neighborhood of $x^*$, 
where $r$ is the rank of 
$\{\nabla h_i(x^*), \nabla g_j(x^*): i \in \mathcal{E}; j \in A(x^*)\}$.
\end{enumerate}
Then, there exists a Lagrange multiplier $(\lambda,\mu)\in\R^m\times\R^p_+$ such that
\begin{equation}
   \label{eqn:wsoc1}
   \nabla f(x^{*})+
   \sum_{i=1}^{m}\lambda_{i}\nabla h_{i}(x^{*})+
   \sum_{j=1}^{p}\mu_{j}\nabla g_{j}(x^{*})=0 \ \ \text{ with } \ \ \mu_{j} g_{j}(x^{*})=0, \forall j, 
\end{equation}
and for every $d \in \mathbb{R}^{n}$ such that 
$\nabla h_i(x^*)^\mathtt{T}d=0$, $\forall i$; $\nabla g_j(x^*)^\mathtt{T}d=0$, $\forall j \in A(x^*)$, 
we have
\begin{equation}
\label{eqn:wsoc2}
        d^{\mathtt{T}}
        (
        \nabla^{2} f(x^{*})+
        \sum_{i=1}^m \lambda_{i} \nabla^{2} h_{i}(x^{*})+
        \sum_{j=1}^p \mu_j \nabla^{2} g_{j}(x^{*})
        ) d\geq 0.
\end{equation}
\end{conjecture}
Note that \eqref{eqn:wsoc1}-\eqref{eqn:wsoc2} is the WSOC condition.
We are aware of two previous attempts of solving this conjecture.
A proof of it under an additional technical condition has appeared, 
recently, in \cite{chineses}. 
Also, a counter-example appeared in \cite{minch}. 
As we will see later in Section 3, these results are incorrect. Also, the recent paper \cite{conjnino2} proved the conjecture for a special form of quadratically-constrained problems.
Our approach is different from the ones mentioned above and 
it is based on an additional assumption of smoothness of the singular value decomposition of $J(x)$ around the basis point $x^*$.

As we have mentioned, WSOC has two important features that makes the Conjecture relevant in practical algorithms, which is our main motivation for pursuing it.
The first one is that it does not rely on the whole set of Lagrange multipliers, 
in contrast with other second-order conditions
in the literature, and 
the second one is that positive semi-definiteness of the 
Hessian of the Lagrangian must be verified in a subspace (a more tractable task) 
rather than at a pointed cone. 

This is compatible with the 
implementation of an algorithm that globally converges to a point $x^*$ fulfilling WSOC.
At each iteration, one has available an aproximation $x^k$ to a solution and a single 
Lagrange multiplier approximation $(\lambda^k, \mu^k)$ and one may check 
if WSOC is approximately satisfied at the current point $(x^k,\lambda^k,\mu^k)$ if one wishes to declare
convergence to a second-order stationary point (see details in \cite{akkt2} and references therein). Of course, this is still a non-trivial computational task, so this only makes sense when most of the effort to check WSOC was already done as part of the computation of the iterate. This is the case of algorithms that try to compute a descent direction and a negative curvature direction \cite{abms2,moguerza}. Near a KKT point, once the procedure for computing the negative curvature direction fails, WSOC is approximately satisfied.

This is an important difference with respect to other conditions that we review in the next section. 
In order to verify an optimality condition that relies on the whole set of Lagrange multipliers, 
one needs an algorithm that generates all multipliers, which may be difficult. 
Even more, in classical second-order conditions, one must check if a matrix is positive semi-definite on a pointed cone, which 
is a far more difficult problem than 
checking it on a subspace (see \cite{murtykabadi}). 
Finally, we are not aware of any 
reasonable iterative algorithm that generates subsequences that converges to a point that satisfies a 
classical, more accurate, second-order optimality condition based on a pointed cone. 
The discussion in \cite{tointexample} indicates that such algorithms probably do not exist.

In Section \ref{sec:opt2} we briefly review some related results on 
second-order optimality conditions. 
In Section 3 we prove the Conjecture under the additional assumption that the singular value decomposition of $J(x)$ is smooth in a neighborhood of $x^*$.
In Section 4 we present some conclusions and future directions of research on this topic.

\section{Second-order optimality conditions}
\label{sec:opt2}
In this section, we review some classical and some recent results on second-order optimility conditions. 
Several second-order optimility conditions have been proposed in the literature, both 
from a theoretical and practical point of view, see  
\cite{bazaraa, rwets, nocedal, fletcher, fiaccomc, bertsekasnl, 
      bshapiro, penotsot, casast, bcshapiro, arupereira, baccaritrad, baccariwcr, gfrererd, bomze}
      and references therein.

First, we start with the basic notation. 
$\mathbb{R}^{n}$ stands for the $n$-dimensional real Euclidean space, $n \in
\mathbb{N}$. $\mathbb{R}_{+}^{n}\subset\mathbb{R}^{n}$
is the set of vector whose components are nonnegative. 
The canonical basis of $\R^n$ is denoted by $e_1,\dots,e_{n}$.
A set $\mathcal{R}\subset \mathbb{R}^{n}$ is a ray if 
$\mathcal{R}:=\{rd_{0}: r \geq 0\}$ for some $d_{0}\in \mathbb{R}^{n}$. 
Given a convex cone $\mathcal{K} \subset \mathbb{R}^{n}$, 
we define the lineality set of $\mathcal{K}$ as $\mathcal{K}\cap-\mathcal{K}$, which is the largest subspace contained in $\mathcal{K}$. 
We say that $\mathcal{K}$ is a first-order cone if $\mathcal{K}$ is the direct sum of a subspace and a ray.

We denote the Lagrangian function by 
$L(x,\lambda,\mu)=f(x)+\sum_{i=1}^m\lambda_i h_i(x)+\sum_{j=1}^p\mu_j g_j(x)$
where $(x, \lambda, \mu)$ is in $\mathbb{R}^{n}\times \mathbb{R}^{m}\times \mathbb{R}_{+}^{p}$
and the generalized Lagrangian function as
$L^{g}(x,\lambda_0,\lambda,\mu)=\lambda_0f(x)+\sum_{i=1}^m\lambda_i h_i(x)+\sum_{j=1}^p\mu_j g_j(x)$ 
where $(x, \lambda_{0}, \lambda, \mu) \in \mathbb{R}^{n}\times \mathbb{R}_{+}\times\mathbb{R}^{m}\times \mathbb{R}_{+}^{p}$.
Clearly, $L^{g}(x,1,\lambda,\mu)=L(x,\lambda,\mu)$.
The symbols $\nabla_{x} L^{g}(x,\lambda_0,\lambda,\mu)$ 
and $\nabla^{2}_{xx} L^{g}(x,\lambda_0,\lambda,\mu)$ stand for 
the gradient and the Hessian of 
$L^{g}(x,\lambda_0,\lambda,\mu)$ with respect to $x$, respectively. 
Similar notation holds for $L(x,\lambda,\mu)$.

The generalized first-order optimality condition at the feasible point $x^{*}$ is
\begin{equation}
\label{eqn:kkt}
 \nabla L^{g}_{x}(x^*,\lambda_0,\lambda,\mu)=0 \ \ \text{ with } \ \ \mu^{\mathtt{T}}g(x^{*})=0, \ \ \lambda_0 \geq0,\ \  \mu\geq0,\ \  
 (\lambda_{0},\lambda, \mu)\neq(0,0,0).
\end{equation}
The set of vectors $(\lambda_0,\lambda,\mu) \in \mathbb{R}_{+}\times\mathbb{R}^{m}\times \mathbb{R}_{+}^{p}$ 
satisfying \eqref{eqn:kkt}
is the set of generalized Lagrange multipliers (or Fritz John multipliers), denoted by $\Lambda_0(x^*)$.
Note that \eqref{eqn:kkt} with $\lambda_0=1$ corresponds to the Karush-Kuhn-Tucker (KKT) conditions, 
the standard first-order condition in numerical optimization. We denote by 
$\Lambda(x^*):=\{(\lambda,\mu) \in \mathbb{R}^{m}\times \mathbb{R}_{+}^{p}: (1,\lambda,\mu) \in \Lambda_0(x^*)\}$, the set of all Lagrange multipliers.
At every minimizer, there are Fritz John multipliers such that \eqref{eqn:kkt} holds, that is, $\Lambda_0(x^*)\neq\emptyset$. In order to get existence of true Lagrange multipliers, additional assumptions have to be required. Assumptions on the analytic description of the feasible set that guarantee the validity of the KKT conditions at local minimizers are called constraint qualification (CQ).
Thus, under any CQ, the KKT conditions are necessary for optimality.

When the second-order information is avaliable, we can consider second-order conditions.
In order to describe second-order conditions (in a dual form), 
we introduce some important sets. We start with the 
cone of critical directions (critical cone), defined as follows:
\begin{equation}
\label{eqn:cone}
 C(x^*):=\{d\in\R^n\mid \nabla f(x^*)^\mathtt{T}d=0; \nabla h_i(x^*)^\mathtt{T}d=0, i\in\mathcal{E}; 
\nabla g_j(x^*)^\mathtt{T}d\leq0, j \in A(x^*)\}.
\end{equation}
Obviouly, $C(x^{*})$ is a non-empty closed convex cone. 
When $\Lambda(x^{*})\neq \emptyset$, the critical cone $C(x^*)$ can be written as
  \begin{equation}\label{eqn:scone}
  \left \{d \in \mathbb{R}^{n} :
            \begin{array}{lll}
            & \nabla h_{i}(x^{*})^{\mathtt{T}}d =0,\text{ for } i \in\mathcal{E},
              \nabla g_{j}(x^{*})^{\mathtt{T}}d =0,\text{ if } \mu_{j}>0\\
            & \nabla g_{j}(x^{*})^{\mathtt{T}}d \leq 0, \text{ if } \mu_{j}=0, j \in A(x^{*})  
            \end{array}
            \right \},
  \end{equation}
for every $(\lambda, \mu) \in \Lambda(x^*)$.   
From the algorithmic point of view, an important set is the critical subspace (or weak critical cone), given by:
\begin{equation}
\label{eq:wcone}
 S(x^*):=\{d\in\R^n\mid \nabla h_i(x^*)^\mathtt{T}d=0, i\in\mathcal{E}; \nabla g_j(x^*)^\mathtt{T}d=0, j \in A(x^*)\}.
\end{equation}
In the case when $\Lambda(x^{*})\neq \emptyset$, a simple inspection shows that the critical subspace $S(x^{*})$ is the lineality space of the critical cone $C(x^*)$.  Under strict complementarity, $S(x^*)$ coincides with $C(x^*)$.

Now, we are able to define the 
classical second-order conditions. 

\begin{definition}
 \label{def:wsoc}
 Let $x^{*}$ be a feasible point with $\Lambda(x^{*})\neq\emptyset$. We have the following definitions
 \begin{enumerate}
  \item 
  We say that the {\it strong second-order optimality condition} (SSOC) holds at $x^{*}$ if
  there is a $(\lambda,\mu)\in\Lambda(x^*)$ such that 
  $d^\mathtt{T}\nabla^2_{xx} L(x^*,\lambda,\mu)d\geq0$ for every $d \in C(x^{*})$.
  \item 
  We say that the {\it weak second-order optimality condition} (WSOC) holds at $x^{*}$ if
  there is  a $(\lambda,\mu)\in\Lambda(x^*)$ such that 
  $d^\mathtt{T}\nabla^2_{xx} L(x^*,\lambda,\mu)d\geq0$ for every $d \in S(x^{*})$.
 \end{enumerate} 
\end{definition}

The classical second-order condition SSOC is particularly important from the point of view of passing from 
necessary to sufficient optimality conditions. In this case, strenghtening the sign of the inequality from ``$\geq0$'' to ``$>0$'' in the definition of SSOC, that is, 
instead of positive semi-definiteness of the Hessian of the Lagrangian on the critical cone, 
we require its positive definiteness 
on the same cone (minus the origin), we get a sufficient optimality condition, see \cite{bazaraa, bshapiro}.
Furthermore, this sufficient condition also ensures that the local minimizer $x^*$ is isolated.
Besides these nice properties, from the practical point of view, SSOC has some disvantages. 
In fact, to verify the validity of SSOC at a given point, is in general, an NP-hard problem, 
\cite{murtykabadi, pardalos}. Also, it is well known that very simple second-order methods fail to generate sequences in which SSOC holds at its accumulation points, see \cite{tointexample}. 

From this point of view, WSOC seems to be the 
most adequate second-order condition 
when dealing with global convergence of second-order methods. In fact, all second-order algorithms known by the authors only guarantee
convergence to points satisfying WSOC, see 
\cite{abms2, bss, cly, conntoint2, dennisv, dennisalem, dipillo, facchinei, grSQP, moguerza} and references therein.

This situation in which a most desirable theoretical property 
is not suitable in an algorithmic framework is not particular only to the second-order case. 
Even in the first-order case, it is known, for example, 
that the Guignard constraint qualification is the 
weakest possible assumption to yield KKT conditions at a local minimizer \cite{gould}. 
In other words, a good first-order necessary optimality condition is of the form ``KKT or not-Guignard''. 
But this is too strong for practical purposes, since no algorithm is 
known to fulfill such condition at limit points of sequences generated by it,
in fact, the convergence assumptions of algorithms require stronger constraint qualifications \cite{rcpld,cpg,amrs2,amrs}.
For second-order algorithms, the situation is quite similar, 
with the peculiarity that the difficulty is not only on the required constraint qualification, 
but also in the verification of the  
optimality condition, since, 
numerically, we can only guarantee a partial second-order property, that is, for directions in the critical subspace, 
which is a subset of the desirable critical cone of directions. 

As the KKT conditions, SSOC and WSOC hold at minimizers only if some additional condition is valid. 
As we will explore in the next section, 
only MFCQ is not enough to ensure the existence of some Lagrange multiplier where SSOC holds. 
Even WSOC can not be assured to hold under MFCQ alone.
Under MFCQ, we have the following result, \cite{bshapiro, bental}:
\begin{theorem} 
\label{teo-bs}
 Let $x^{*}$ be a local minimizer of \eqref{problem}. Assume that MFCQ holds at $x^{*}$. Then, 
 \begin{equation}
  \label{bonnans-shapiro1}
 \mbox{for each } d\in C(x^*),\mbox{ there is a multiplier }(\lambda,\mu)\in\Lambda(x^*)\mbox{ such that }
  d^\mathtt{T}\nabla^2_{xx} L(x^*,\lambda,\mu)d\geq0.
 \end{equation}
\end{theorem}

Note that for each critical direction, we have an associated Lagrange multiplier 
$(\lambda,\mu)\in\Lambda(x^*)$, in opposition to 
SSOC or WSOC, where we require {\it the same} Lagrange multiplier for all critical directions. 
Observe that \eqref{bonnans-shapiro1} does not imply WSOC (and neither SSOC).

Observe also that since $\Lambda(x^{*})$ is a compact set (by MFCQ), 
\eqref{bonnans-shapiro1} can be written in a more compact form, namely, 
$$
\forall d\in C(x^*), \ \ \text{sup} \{d^\mathtt{T}\nabla^2_{xx} L(x^*,\lambda,\mu)d : (\lambda,\mu)\in\Lambda(x^*))\}\geq0.
$$ Although this optimality condition relies on the whole Lagrange multiplier set $\Lambda(x^*)$, 
hence it is not suitable for our practical considerations, it will play a crucial role in our analysis.

Even when no constraint qualification is assumed,
a second-order optimality condition can be formulated, relying on Fritz John multipliers \eqref{eqn:kkt}:

\begin{theorem}
\label{teo-aru}
Let $x^*$ be a local minimizer of \eqref{problem}. Then, for every $d$ in the critical cone $C(x^*)$, there is a Fritz John 
multiplier $(\lambda_0,\lambda,\mu)\in\Lambda_0(x^*)$ such that 
 \begin{equation}
 \label{arutyunov}
 d^\mathtt{T}\nabla^2_{xx} L^{g}(x^*,\lambda_0,\lambda,\mu)d\geq0.
 \end{equation}
\end{theorem}

The optimality condition of Theorem \ref{teo-aru} 
has been studied a lot over the years, \cite{dubo, bental, levitin, bshapiro, aru}. 
An important property is that it can be transformed into a sufficient optimality condition 
by simply replacing the non-negative sign ``$\geq0$'' by ``$>0$'' (except at the origin), without any additional assumption. 
For this reason, this condition is said to be a ``no-gap'' optimality condition. 
Note that this is different from the case of SSOC, 
since an additional assumption must be made for the necessary condition to hold. 
Note that Theorem \ref{teo-bs} can be derived from Theorem \ref{teo-aru}, since under MFCQ, 
there is no Fritz John multiplier with $\lambda_0=0$.

We emphasize that even though optimality conditions 
given by Theorems \ref{teo-bs} and \ref{teo-aru} have nice theoretical properties, they do not suit our framework since 
their verification requires the knowledge of the whole set of (generalized) Lagrange multipliers 
at the basis point, whereas in practice, 
we only have access to (an approximation of) a single Lagrange multiplier. 
In the case of the optimality condition given by Theorem \ref{teo-aru}, one could argue that the possibility of verifying it with $\lambda_0=0$, and hence independently of the objective function, is not useful at all as an optimality condition. This is arguably the case for the first-order Fritz John optimality condition, but since Theorem \ref{teo-aru} gives a ``no-gap'' optimality condition, this argument is not convincent in the second-order case. In fact, one could show that if the sufficient optimality condition associated to Theorem \ref{teo-aru} is fulfilled with $\lambda_0=0$ for all critical directions, then the basis point is an isolated feasible point, and hence a local solution independently of the objective function. We take the point of view that algorithms naturally treat differently the objective function and the constraint functions, in a way that a multiplier associated to the objective function is not present, hence our focus on Lagrange multipliers, rather than on Fritz John multipliers.

As we have mentioned, known practical methods are only guaranteed to converge to points satisfying WSOC, and hence, we focus our attention, from now on, on conditions ensuring it at local minimizers.

We start with \cite{baccaritrad}, where the authors investigate the issue of verifying (\ref{bonnans-shapiro1}) 
for the same Lagrange multiplier: 

\begin{theorem}[\cite{baccaritrad}]
\label{bttheorem}
Let $x^*$ be a local minimizer of (\ref{problem}). Assume that MFCQ holds at $x^*$ 
and that $\Lambda(x^*)$ is a line segment.
Then, for every first-order cone $K\subset C(x^*)$, 
there is a $(\lambda^K,\mu^K)\in\Lambda(x^*)$ such that 
\begin{equation}
\label{bt}
\forall d\in K, \ \  d^\mathtt{T}\nabla^2_{xx} L(x^*,\lambda^K,\mu^K)d\geq0.
\end{equation}
\end{theorem}

We are interested only in the special case $K:=S(x^*)$. Thus, 
(\ref{bt}) holds at a local minimizer $x^*$ 
when $\Lambda(x^*)$ is a line segment and MFCQ holds at $x^{*}$ (or, equivalently, $\Lambda(x^*)$ is a bounded line segment). 
Note that in this case, \eqref{bt} is equivalent to WSOC. 

In order to prove Theorem \ref{bttheorem} a crucial result is Yuan's Lemma \cite{yuan}, 
which was generalized for first-order cones in \cite{baccaritrad}. For further applications of 
Yuan's Lemma, see \cite{martinezyuan, crouzeixyuan}.

\begin{lemma}[Yuan \cite{yuan,baccaritrad}] 
\label{yuan}
Let $P,Q\in\R^{n\times n}$ be two symmetric matrices and $K\subset\R^n$ a first-order cone. 
Then the following conditions are equivalent:
\begin{itemize}
\item $\max\{d^\mathtt{T}Pd,d^\mathtt{T}Qd\}\geq0, \ \ \forall d\in K$;
\item There exist $\alpha\geq0$ and $\beta\geq0$ with $\alpha+\beta=1$ such that $d^\mathtt{T}(\alpha P+\beta Q)d\geq0$, $\forall d\in K$.
\end{itemize}
\end{lemma}

A sufficient condition to guarantee that $\Lambda(x^*)$ is a line segment, is to require that the 
rank of the Jacobian matrix $J(x^*)$ is row-deficient by at most one, that is, the rank is one less than the number of rows. 
The fact that the rank assumption yields the one-dimensionality of $\Lambda(x^*)$ is a simple consequence of 
the rank-nullity theorem. 
Thus, we have the following result:
\begin{theorem}[Baccari and Trad \cite{baccaritrad}] 
\label{bttheo}
Let $x^*$ be a local minimizer of \eqref{problem} such that MFCQ holds 
and the rank of the Jacobian matrix $J(x^*)\in\R^{(m+q)\times n}$ is $m+q-1$, 
where $q$ is the number of active inequality constraints at $x^*$. 
Then, there exists a Lagrange multiplier $(\lambda,\mu)\in\R^m\times\R^p_+$ such that 
WSOC holds.
\end{theorem}


Another line of reasoning in order to arrive at second-order optimality conditions 
is to use Janin's version of the classical Constant Rank theorem (\cite{spivak}, Theorem 2.9). 
See \cite{janin, aes2, param}.

\begin{theorem}[Constant Rank]
\label{ranktheo}
Let $x^*\in\Omega$ and $d\in C(x^*)$. 
Let $E\subset\{1,\dots,p\}$ be the set of indices $j$
such that $\nabla g_j(x^*)^\mathtt{T}d=0, j\in A(x^*)$. 
If $\{\nabla h_i(x), i\in\mathcal{E}; \nabla g_j(x), j\in E\}$ 
has constant rank in a neighborhood of $x^*$, then, there are $\varepsilon>0$ and a twice continuously differentiable 
function $\xi:(-\varepsilon,\varepsilon)\to\R^n$  
such that $\xi(0)=x^*, \xi'(0)=d, h_i(\xi(t))=0, i\in\mathcal{E}; g_j(\xi(t))=0, j\in E$ for 
$t\in(-\varepsilon,\varepsilon)$ and $g(\xi(t))\leq0$ for $t\in[0,\varepsilon)$.
\end{theorem}

The proof that the function $\xi$ is twice continuously differentiable was done in \cite{param}.\\

Using a constant rank assumption jointly with MFCQ, in \cite{ams2}, 
Andreani, Mart\'{\i}nez and Schuverdt have proved the existence of multipliers satisfying WSOC 
at a local minimizer as stated below. This joint condition was also used in the convergence 
analysis of a second-order augmented Lagrangian method. 

\begin{theorem}[Andreani, Mart\'{\i}nez and Schuverdt \cite{ams2}]
\label{ams2}
Let $x^*$ be a local minimizer of \eqref{problem} with MFCQ holding at $x^*$. 
Assume that the rank of the Jacobian matrix $J(x)\in\R^{(m+q)\times n}$ 
is constant around $x^*$, where $q$ is the number of active inequality constraints at $x^*$.
Then, WSOC holds at $x^{*}$.
\end{theorem}

The proof can be done using Theorem \ref{ranktheo} for $d\in S(x^*)$
and $E=\{1,\dots,p\}$, using the fact that $t=0$ is a local minimizer of $f(\xi(t)), t\geq0$.

This result was further improved in \cite{aes2}, where they noticed that MFCQ can be replaced by the non-emptyness of $\Lambda(x^*)$. This was also done independently in \cite{jye}. In fact, WSOC can be proved to hold {\it for all} Lagrange multipliers:

\begin{theorem}[Andreani, Echag\"ue and Schuverdt \cite{aes2}]
\label{aes}
Let $x^*$ be a local minimizer of \eqref{problem} such that the rank of the Jacobian matrix $J(x)\in\R^{(m+q)\times n}$ 
is constant 
around $x^*$, where $q$ is the number of active inequality constraints at $x^*$.
Then, every Lagrange multiplier $(\lambda,\mu)\in\R^m\times\R^p_+$ (if any exists) is such that WSOC holds.
\end{theorem}

This same technique can be employed under the Relaxed Constant Rank CQ (RCRCQ, \cite{minchenko}), that is,
$\{\nabla h_i(x),i\in\mathcal{E}; \nabla g_j(x), j\in E\}\mbox{ has constant rank around }x^*\mbox{ for every }E\subset A(x^*),$
to prove the stronger result that all Lagrange multipliers satisfy SSOC. See \cite{aes2,param}.
These results can be strengthened by replacing the use of the Constant Rank theorem by the assumption that the critical cone is a subset of the Tangent cone of a modified feasible set (Abadie-type assumptions). See details in \cite{abadie2,bomze}.

\section{The conjecture}

In this section we prove the conjecture under an additional assumption based on the smoothness of the singular value decomposition of the Jacobian matrix.
In view of Theorems \ref{bttheo} and \ref{ams2}, 
that arrives at the same second-order optimality condition under MFCQ and 
row-rank deficiency of at most one or under MFCQ and the constant rank assumption of the Jacobian $J(x)$, 
it is natural to conjecture that the same result would hold under MFCQ and 
assuming that the rank increases at most by one in a neighborhood of a local minimizer. 
This was conjectured in \cite{ams2}. 
Although an unification of both results would be interesting, 
this was a bold conjecture since the theorems have completely different proofs.

Let us first show that Baccari and Trad's result can be generalized 
in order to consider column-rank deficiency. 
The proof is a simple application of the rank-nullity theorem.

\begin{theorem}
\label{genbac}
Let $x^*$ be a local minimizer of \eqref{problem} such that MFCQ holds and the 
rank of the Jacobian matrix $J(x^*)\in\R^{(m+q)\times n}$ is $n-1$, where $q$ is the number of active inequality constraints at $x^*$. 
Then, there exists a Lagrange multiplier $(\lambda,\mu)\in\R^m\times\R^p_+$ such that WSOC holds.
\end{theorem}
\begin{proof} 
Applying the rank-nullity theorem to $S(x^*)=\text{Ker}(J(x^*))$, we get that 
$\mbox{dim}(S(x^*))=1$. 
Hence, there is $d_0\in S(x^*)$ such that $S(x^*)=\{td_0, t\in\R\}$. 
Since MFCQ holds, Theorem \ref{teo-bs} yields (\ref{bonnans-shapiro1}).  
In particular, for $d=d_{0}$, there is a Lagrange multiplier $(\lambda,\mu)$ such that $d_0^\mathtt{T}\nabla^2_{xx} L(x^*,\lambda,\mu)d_0\geq0$. 
This same Lagrange multiplier can be used for all other directions $d=td_0\in S(x^*)$, 
since $d^\mathtt{T}\nabla^2_{xx} L(x^*,\lambda,\mu)d=t^2d_0^\mathtt{T}\nabla^2_{xx} L(x^*,\lambda,\mu)d_0\geq0$.
Thus, WSOC holds at $x^{*}$.
\end{proof}

The previous results show that the Conjecture is true in dimension less than or equal to two, 
or when there are at most two active constraints. 
In $\R^3$, the remarkable example by Arutyunov \cite{aru}/Anitescu \cite{ani} shows that if the rank increases by more than two around $x^*$, 
WSOC may fail for all Lagrange multipliers (also, SSOC fails).

We describe below a modification of the example, given in \cite{Baccari2004}, 
since it gives nice insights about the Conjecture.
\begin{example}
$$\begin{array}{ll}
\mbox{Minimize }&x_3,\\
&x_3\geq 2\sqrt{3}x_1x_2-2x_2^2,\\
&x_3\geq x_2^2-3x_1^2,\\
&x_3\geq -2\sqrt{3}x_1x_2-2x_2^2.\\
\end{array}$$
\end{example}
Here $x^*=(0,0,0)$ is a global minimizer. The critical subspace is the whole plane $x_3=0$ and $\Lambda(x^*)$ is 
the simplex $\mu_1+\mu_2+\mu_3=1, \mu_j\geq0, j=1,2,3$. 
Figure 1 shows the graph of the right-hand side of each constraint, 
where the feasible set is the set of points above all surfaces. 
Note that along every direction in the critical cone, 
there is a convex combination of the constraints that moves upwards and (\ref{bonnans-shapiro1}) holds, 
but for any convex combinations of the constraints, 
there exists a direction in the critical cone that moves downwards. 
This means that WSOC fails for all Lagrange multipliers. Since in this 
example $C(x^*)=S(x^*)$, this means that SSOC also fails for all Lagrange multipliers. 
Note also in Figure 1 that around $x^*$ there is no feasible curve such that all 
constraints are active along this curve, which is the main property allowing 
the proof of WSOC under constant rank assumptions (Theorem \ref{ranktheo}).

\begin{figure}[h]
\label{baccari}
\begin{center}
\includegraphics[scale=0.7]{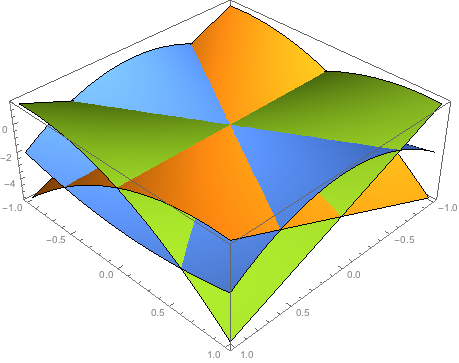}
\caption{MFCQ alone is not enough to ensure the validity of SSOC or WSOC.}
\end{center}
\end{figure}

Before describing our proof, we briefly point out previous attempts of solving the Conjecture. 
In \cite{chineses}, the authors stated the validity of the Conjecture with the additional following assumption:\\

{\bf Assumption (A3) \cite{chineses}}: \\
If there exists a sequence $\{x^k\}$ converging to $x^*$ 
such that the rank of $J(x^k)$ is $r+1$ for all $k$, then for any $x^k$ and any subset $E\subset A(x^*)$, 
the rank of $\{\nabla h_i(x), i\in\mathcal{E}; \nabla g_j(x), j\in E\}$ is constant around $x^k$.\\

The proof of the Conjecture under (A3) in \cite{chineses} is based on the following incorrect Lemma:\\

{\bf Lemma 3.4 from \cite{chineses}}: Under (A3) and the assumptions of the Conjecture, 
whenever there exists a sequence $\{x^k\}$ converging to $x^*$ such that the rank of $J(x^k)$ is $r+1$ 
for all $k$, there exists some index set $E\subset A(x^*)$ such that the rank of $\{\nabla g_j(x^k), j\in E\}$ 
is equal to $1$ and the rank of $\{\nabla h_i(x), i\in\mathcal{E}; \nabla g_j(x), j\in A(x^*)\backslash E\}$ 
is $r$ for an infinite number of indices $k$.

The following counter-example shows that it is incorrect.

\begin{counter}
$$\begin{array}{ll}
\mbox{Minimize }&0,\\
\mbox{subject to }&x_1\leq0,\\
&x_1+x_2x_1^2+x_2^3/3\leq0,\\
&2x_1+x_2x_1^2+x_2^3/3\leq0,
\end{array}$$
at $x^*=(0,0)$. 
\end{counter}

Clearly, $x^*=(0,0)$ is a local minimizer that fulfills MFCQ, 
the rank of the Jacobian is $1$ at $x^*$ and at most $2$ at every other point.
Also, for any subset of $A(x^*)=\{1,2,3\}$, the rank of the associated gradients
is constant in the neighborhood of every point $x$ different from $x^*$, hence all 
assumptions of Lemma 3.4 from \cite{chineses} are fulfilled. 
A simple inspection shows that for every point $x$ different from $x^*$ one can not separate the three
gradients into two subsets of rank $1$ as the lemma states. In fact, every subset with two
gradients will have rank $2$. 

Another attempt to solve the Conjecture is given in \cite{minch}.
Here, the following problem is presented as a counter-example for the Conjecture:
\begin{equation*}
            \begin{array}{lll}
           \mbox{minimize   } & -x_1^2-x_2, \\ 
           \mbox{subject to } & 2x_1^2+x_2\leq0,\\
           & -x_1^2+x_2\leq0,\\
           & x_2\leq0.                            
            \end{array}
\end{equation*}
The point $x^*:=(0,0)$ is a local minimizer that satisfies MFCQ, the rank of $J(x^*)$ is $1$ and increases at most to $2$ around $x^*$. In \cite{minch}, the authors only show that WSOC does not need to hold for all Lagrange multipliers. In particular, it was shown that it does not hold at $\bar{\mu}:=(0,1,0)$. 
Clearly, it does not disprove the Conjecture since there are other Lagrange multipliers that fulfill WSOC, as implied by Theorem \ref{genbac}. For instance, $\mu:=(1,0,0)$ satisfies WSOC.

Finally, we present our main result. 
We prove the Conjecture under an additional technical assumption on the smoothness of the singular value decomposition (SVD) of the Jacobian matrix around $x^*$.

\begin{assumption}
\label{svd}
Let $q$ be the number of active inequality constraints at $x^*$ and $J(x)\in\R^{(m+q)\times n}$ be the Jacobian matrix for $x$ near $x^*$.
We assume that there exist differentiable functions around $x^*$ given by 
$$x\mapsto U(x)\in\R^{(m+q)\times(m+q)}, \ \ x\mapsto\Sigma(x)\in\R^{(m+q)\times n}\mbox{ and }x\mapsto V(x)\in\R^{n \times n},$$
such that $J(x)=U(x)\Sigma(x)V(x)^{\mathtt{T}}$, where $\Sigma(x)$ is diagonal with diagonal elements $\sigma_1(x),\sigma_2(x),\dots,\sigma_k(x)$, 
where $k:=\min\{m+q,n\}$ and $\sigma_i(x)=0$ when $i$ is greater than the rank of $J(x)$. We assume also that $U(x^*)$ and $V(x^*)$ are matrices with non-zero orthogonal columns.
\end{assumption}

Note that only at $x=x^*$ we assume that  $U(x^*)$ and $V(x^*)$ are matrices with orthogonal columns. This implies that $U(x)$ and $V(x)$ are at least invertible matrices in a small enough neighborhood of $x^*$, but not necessarily with orthogonal columns.

Note that we do not require that the columns of the matrices $U(x^*)$ and $V(x^*)$ to be normalized, 
as in the classical SVD decomposition. In our proof, it is also not necessary to adopt the convention of non-negativeness of $\sigma_i(x)$ or that they are ordered. Without risk of confusion, we will still call this weaker decomposition as the SVD. We introduce this extra freedom in the decomposition in order to allow more easily for the differentiability of the functions.

\begin{theorem}\label{theo} Assume that $x^*$ is a local minimizer of (\ref{problem}) 
that fulfills MFCQ. Let $r$ be the rank of $J(x^*)$, and assume that for every $x$ in 
some neighborhood of $x^*$ the rank of $J(x)$ is at most $r+1$. 
Suppose also that Assumption \ref{svd} holds. Then, there is a Lagrange multiplier ($\lambda,\mu)$ such that WSOC holds.
\end{theorem}
\begin{proof}
Let us consider the column functions 
$U(x)=[u_1(x) \dots u_{m+q}(x)]$ and $V(x)=[v_1(x) \dots v_n(x)]$. 
Clearly, $J(x)=\sum_{k=1}^{r+1}\sigma_k(x)u_k(x)v_k(x)^\mathtt{T}$. 
To simplify the notation, let us assume that $m=0$ and $A(x^*)=\{1,\dots,p\}$, 
that is, $q=p$, hence, it holds that

$$\nabla g_i(x)^\mathtt{T}=\sum_{k=1}^{r+1}\sigma_k(x)[u_k(x)]_iv_k(x)^\mathtt{T}, \ \ i=1,\dots,p,$$

where $[u]_i$ is the $i$-th coordinate of the vector $u$ of 
appropriate dimension. 

Since the functions are smooth we can compute derivatives and get

$$\nabla^2 g_i(x)=
  \sum_{k=1}^{r+1} \sigma_k(x)[u_k(x)]_iJ_{v_k}(x)+
  \sum_{k=1}^{r+1}([u_k(x)]_i\nabla\sigma_k(x)+\sigma_k(x)\nabla[u_k(x)]_i)v_k(x)^\mathtt{T}, \ \ i=1,\dots,p,$$

where $J_{v_k}(x)\in\R^{n\times n}$ is the Jacobian matrix of the function $v_k$ at $x$. 

Now, let us fix a direction $d\in S(x^*)$ and a Lagrange multiplier 
$\mu \in\Lambda(x^*)$ (we identify a Lagrange multiplier $(\lambda,\mu)$ with $\mu$ since we are assuming $m=0$).
Now, we proceed to evaluate $d^\mathtt{T}\nabla^2_{xx}L(x^*,\mu)d$. We omit the dependency on $x^*$ for simplicity. Then, 

\begin{eqnarray}
 \begin{array}{rl}
 d^\mathtt{T}\nabla^2_{xx}L(x^*,\mu)d 
 =&d^\mathtt{T}\nabla^2 fd+d^\mathtt{T}\left[\sum_{i=1}^p\mu_i\sum_{k=1}^{r+1}([u_k]_i\nabla\sigma_k+\sigma_k\nabla[u_k]_i)v_k^\mathtt{T}+\sigma_k[u_k]_iJ_{v_k}\right]d \\
 =&d^\mathtt{T}\nabla^2fd+
   \sum_{k=1}^{r+1}[(d^\mathtt{T}\nabla\sigma_k)(\mu^\mathtt{T}u_k)+\sigma_kd^\mathtt{T}J_{u_k}^\mathtt{T}\mu]v_k^\mathtt{T}d+
   \sum_{k=1}^{r+1}\sigma_k(\mu^\mathtt{T}u_k)d^\mathtt{T}J_{v_k}d.
 \end{array}  
\end{eqnarray}

From $S(x^*)=\text{Ker}(J(x^*))$, 
and from the SVD, $J(x)=U(x)\Sigma(x)V(x)^\mathtt{T}$, we can conclude that 
there are $s_{r+1},\dots,s_{n}\in\R$ such that $d=\sum_{j=r+1}^ns_jv_j$. Hence, 
from the orthogonality of $\{v_1,\dots,v_p\}$, 
we get $v_{k}^{\mathtt{T}}d=0$, $k<r+1$. 
Furthermore, since $\sigma_{r+1}=0$ we obtain that

\begin{equation}
 \label{eqn:hess}
  d^\mathtt{T}\nabla^2_{xx}L(x^*,\mu)d=
  d^\mathtt{T}\nabla^2fd+
  (d^\mathtt{T}\nabla\sigma_{r+1})(\mu^\mathtt{T}u_{r+1})(v_{r+1}^\mathtt{T}d)+
  \sum_{k=1}^{r}\sigma_k(\mu^\mathtt{T}u_k)d^\mathtt{T}J_{v_k}d.
\end{equation}

For a fixed Lagrange multiplier $\bar{\mu} \in\Lambda(x^*)$ 
(note that MFCQ ensures non-emptyness of $\Lambda(x^*)$), we can write $\Lambda(x^*)=(\bar{\mu}+\text{Ker}(J(x^*)^\mathtt{T}))\cap\R^p_+$, hence, there are $t_{r+1},\dots,t_p\in\R$ such that $\mu=\bar{\mu}+\sum_{j=r+1}^{p} t_ju_j$ and we can 
write \eqref{eqn:hess} as

\begin{equation}
\label{eqn:hesst}
d^\mathtt{T}\nabla^2_{xx}L(x^*,\mu)d=
d^\mathtt{T}\nabla^2fd+(d^\mathtt{T}\nabla\sigma_{r+1})(\bar{\mu}^\mathtt{T}u_{r+1}+t_{r+1})(v_{r+1}^\mathtt{T}d)+
\sum_{k=1}^{r}\sigma_k(\bar{\mu}^\mathtt{T}u_k)d^\mathtt{T}J_{v_k}d.
\end{equation}
Observe that for a fixed $d\in S(x^*)$, the value of $d^\mathtt{T}\nabla^2_{xx} L(x^*,\mu)d$, for $\mu \in\Lambda(x^*)$, depends on a single parameter $t_{r+1}$. 
Since MFCQ holds, condition (\ref{bonnans-shapiro1}) holds, and we may write it as

\begin{equation}
\label{aru2}
\max_{\mu \in \Lambda(x^{*})} d^\mathtt{T}\nabla^2_{xx} L(x^*,\mu)d \geq0, \ \ \forall d\in S(x^*).
\end{equation}

In virtue of the fact that the value of $d^\mathtt{T}\nabla^2_{xx}L(x^*,\mu)d$ depends only on one parameter, 
in order to apply Yuan's lemma, we will rewrite \eqref{aru2} as a maximization problem over a line segment.
For that purpose, we define the set
\begin{equation}
\label{eqn:m}
M:=\{(t_{r+1},t_{r+2},\dots,t_{p}) \in \mathbb{R}^{p-r} \text{ such that } 
\bar{\mu}+\sum_{j=r+1}^{p} t_{j}u_{j} \in \Lambda(x^{*}) \},
\end{equation}
and the following optimization problems
\begin{equation}
 a_{*}:= \text{ Inf } \{ t_{r+1}:(t_{r+1},t_{r+2},\dots,t_{p}) \in M \},
\end{equation}
and 
\begin{equation}
 b_{*}:= \text{ Sup } \{ t_{r+1}:(t_{r+1},t_{r+2},\dots,t_{p}) \in M \}.
\end{equation}
Both values $a_{*}$ and $b_{*}$ are finite and attained since $M$ is a compact set.
Furthermore, $M$ is a compact convex set. It is easy to see, that $M$ is convex and closed since
$\Lambda(x^{*})$ is also convex and closed. To show that $M$ is bounded, suppose by contradiction, that there is a sequence
$(t_{r+1}^{k},\dots,t_{p}^{k}) \in M$ with $T_{k}:=\max\{|t^{k}_{r+1}|,\dots,|t_{p}^{k}|\} \rightarrow \infty$.
Since $\Lambda(x^{*})$ is bounded, there is a scalar $K$ such that 
 $\|\bar{\mu}+\sum_{j=r+1}^{p} t^{k}_{j}u_{j} \| \leq K$, 
 $\forall k \in \mathbb{N}$.
Dividing this expression by $T_{k}$ and taking an adequate subsequence, 
we conclude that there are $\bar{t}_{r+1},\dots, \bar{t}_{p}$ not all zero, such that 
$\sum_{j=r+1}^{p} \bar{t}_{j}u_{j}=0$, which is a contradiction with the linear independence of $\{u_{r+1},\dots,u_{p}\}$.
Thus, $M$ must be a compact convex set.

Finally, denote by
$\theta(t_{r+1},d):=d^\mathtt{T}\nabla^2fd+
                     (\bar{\mu}^\mathtt{T}u_{r+1}+t_{r+1})(d^\mathtt{T}\nabla\sigma_{r+1})(v_{r+1}^\mathtt{T}d)+
\sum_{k=1}^{r}\sigma_k(\bar{\mu}^\mathtt{T}u_k)d^\mathtt{T}J_{v_k}d$. 
We see from \eqref{eqn:hesst}, 
that 
$\theta(t_{r+1},d)$ coincides with 
$d^\mathtt{T}\nabla^2_{xx}L(x^*,\mu)d$ whenever $\mu=\bar{\mu}+\sum_{j=r+1}^{p}t_{j}u_{j}$.

Now, consider the optimization problem: 
\begin{equation}
\label{eqn:theta}
 \max \{ \theta(t_{r+1},d): t_{r+1} \in [a_{*},b_{*}]\}.
\end{equation}
From \eqref{aru2}, we get $\max \{ \theta(t_{r+1},d): t_{r+1} \in [a_{*},b_{*}]\}\geq0$ for all $d \in S(x^{*})$.
Observe that $\theta(t,d)$ is linear in $t$ and for each $t$ fixed, it defines 
a quadratic form as function of $d$.
Then, since  $\theta(t,d)$ is linear in $t$, 
the maximum of \eqref{eqn:theta} is attained either at $t_{r+1}=a_{*}$ or $t_{r+1}=b_{*}$.
For simplicity, let us call the quadratic forms $\theta(a_{*},d)$ and
$\theta(b_{*},d)$ by $d^\mathtt{T}Pd$ and $d^\mathtt{T}Qd$, respectively.
Thus, we arrive at
\begin{equation}
 \label{eqn:yuan2}
 \max\{ d^\mathtt{T}Pd, d^\mathtt{T}Qd\}\geq0,\ \  \forall d\in S(x^*).
\end{equation}

Applying Yuan's Lemma (Lemma \ref{yuan}), 
we get the existence of $\alpha\geq0$ and 
$\beta\geq0$ with $\alpha+\beta=1$ such that 

\begin{equation}
\label{eqn:yuan3} 
 d^\mathtt{T}(\alpha P+\beta Q)d\geq0, \ \ \forall d\in S(x^*). 
\end{equation}

Additionally, due to the linearity of $\theta(t,d)$, we see that $\eta:=\alpha a_{*}+\beta b_{*}$ 
satisfies $\theta(\eta,d)=d^\mathtt{T}(\alpha P+\beta Q)d$. 
Denote $\pi_{r+1}(t_{r+1},\dots,t_{p}):=t_{r+1}$ the projection
onto the first coordinate. From the continuity of $\pi_{r+1}$ and the compactness and convexity of $M$, we get $[a_{*}, b_{*}]=\pi_{r+1}(M)$. Since $\eta \in[a_{*}, b_{*}]$, we conclude that there are some scalars $\hat{t}_{r+1},\dots, \hat{t}_{p}$ with 
$\hat{t}_{r+1}=\eta$ and $\hat{\mu}:=\bar{\mu}+\sum_{j=r+1}^{p} \hat{t}_{j}u_{j} \in \Lambda(x^*)$
such that $\theta(\eta,d)=d^\mathtt{T}(\alpha P+\beta Q)d=d^\mathtt{T}\nabla^2_{xx} L(x^*,\hat{\mu})d$. 
From \eqref{eqn:yuan3}, we get 
that $\hat{\mu}$ is a Lagrange multiplier for which WSOC holds at $x^*$. 
\end{proof}


We refer the reader to the companion paper \cite{smoothSVD} for the proof that our Theorem \ref{theo} is a generalization of Theorems 
\ref{bttheo} (originally from Baccari and Trad \cite{baccaritrad}) and \ref{genbac}, 
where the rank deficiency is at most one. It is also a generalization of Theorem \ref{ams2} (originally from Andreani et al. \cite{ams2}), where the rank is constant, as long as all non-zero singular values are distinct.

In order to see that Theorem \ref{theo} provides new examples where WSOC holds, let us consider the following:
\begin{example}
$$\begin{array}{ll}
\mbox{Minimize }&x_3,\\
\mbox{subject to }&\cos(x_1+x_2)-x_ 3-1\leq0,\\
&-\cos(x_1+x_2)-x_3+1\leq0,\\
&-2x_3\leq0,
\end{array}$$
at a local minimizer $x^*=(0,0,0)$.
The Jacobian matrix for $x$ around $x^*$ is given by 
$$J(x)=\left(\begin{array}{ccc}-\sin(x_1+x_2)&-\sin(x_1+x_2)&-1\\\sin(x_1+x_2)&\sin(x_1+x_2)&-1\\0&0&-2\end{array}\right),$$ that admits the following smooth SVD decomposition:
\begin{eqnarray*}
U(x):=\left(\begin{array}{ccc}-\frac{1}{\sqrt{6}}&-\frac{1}{\sqrt{2}}&\frac{1}{\sqrt{3}}\\-\frac{1}{\sqrt{6}}&\frac{1}{\sqrt{2}}&\frac{1}{\sqrt{3}}\\-\frac{2}{\sqrt{6}}&0&-\frac{1}{\sqrt{3}}\end{array}\right), \Sigma(x):=\left(\begin{array}{ccc}\sqrt{6}&0&0\\0&2\sin(x_1+x_2)&0\\0&0&0\end{array}\right)\mbox{ and}\\
V(x)^\mathtt{T}:=\left(\begin{array}{ccc}0&0&1\\\frac{1}{\sqrt{2}}&\frac{1}{\sqrt{2}}&0\\\frac{1}{\sqrt{2}}&-\frac{1}{\sqrt{2}}&0\end{array}\right).
\end{eqnarray*} 
Clearly, MFCQ holds and the rank is $1$ at $x^*$ and increases at most to $2$ around $x^*$. 
The set $\Lambda(x^*)$ of Lagrange multipliers is defined by the relations $\mu_1+\mu_2+2\mu_3=1$ with $\mu_1,\mu_2,\mu_3\geq0$. 
Theorem \ref{theo} guarantees the existence of a Lagrange multiplier that fulfill WSOC. 
In fact, we can see that WSOC holds at $x^*$ whenever $\mu_2\geq\mu_1$.\end{example}

The next example shows, however, that our Theorem \ref{theo} in the way presented does not prove the complete conjecture, since Assumption \ref{svd} may fail.

\begin{example}
$$\begin{array}{ll}
\mbox{Minimize }&x_3,\\
\mbox{subject to }&g_1(x):=x_1x_2-x_ 3\leq0,\\
&g_2(x):=-x_1x_2-x_3\leq0,\\
&g_3(x):=-x_3\leq0,
\end{array}$$
at a local minimizer $x^*=(0,0,0)$.

The jacobian matrix at $x$ near $x^*$ is given by $J(x)=\left(\begin{array}{ccc}x_2&x_1&-1\\-x_2&-x_1&-1\\0&0&-1\end{array}\right).$
Clearly, MFCQ holds and the rank of $J(x^*)$ is $1$ and increases at most to $2$ in a neighborhood. Also, WSOC holds. Let us prove that Assumption \ref{svd} does not hold. Assume that there are differentiable functions $U(x), \Sigma(x), V(x)$ such that $J(x)=U(x)\Sigma(x)V(x)^\mathtt{T}$ for all $x$ in a neighborhood of $x^*$ as in Assumption \ref{svd}. Let $U(x)=[u_1, u_2, u_3]$ and $V(x)=[v_1,v_2,v_3]$ be defined columnwise, where the dependency on $x$ was omitted. Also, let $\sigma_1, \sigma_2, \sigma_3$ be the diagonal elements of $\Sigma(x)$. Since the rank of $J(x)$ is at most $2$, $\sigma_3\equiv0$. At $x=x^*$, since the rank is $1$, we have $\sigma_1(x^*)\neq0$ and $\sigma_2(x^*)=0$. Then, $u_1(x^*)=\frac{\alpha}{\sigma_1(x^*)}(-1,-1,-1), v_1(x^*)=\frac{1}{\alpha}(0,0,1)$ for some $\alpha\neq0$ and $u_2(x^*)\perp u_1(x^*)$, $v_2(x^*)\perp v_1(x^*)$. Now, denoting by $[w]_i, i=1,2,3$ the components of the vector $w\in\R^3$, the first two columns of the identity $U(x)\Sigma(x)V(x)^\mathtt{T}=J(x)$ for all $x$ near $x^*$  gives:
$$\sigma_1[u_1]_i[v_1]_j+\sigma_2[u_2]_i[v_2]_j=\frac{\partial g_i(x)}{\partial x_j}, i=1,2,3, j=1,2.$$
Computing derivatives with respect to $x_1$ and $x_2$ of every entry, gives, for $i=1,2,3, j=1,2$ and $k=1,2$:
\begin{eqnarray*}
\frac{\partial\sigma_1}{\partial x_k}[u_1]_i[v_1]_j+\sigma_1\frac{\partial[u_1]_i}{\partial x_k}[v_1]_j+\sigma_1[u_1]_i\frac{\partial[v_1]_j}{\partial x_k}+\\
\frac{\partial\sigma_2}{\partial x_k}[u_2]_i[v_2]_j+\sigma_2\frac{\partial[u_2]_i}{\partial x_k}[v_2]_j+\sigma_2[u_2]_i\frac{\partial[v_2]_j}{\partial x_k}=
\frac{\partial^2 g_i(x)}{\partial x_k\partial x_j}.
\end{eqnarray*}
At $x=x^*$, since $\sigma_2(x^*)=0$ and $[v_1(x^*)]_j=0, j=1,2$ we have
\begin{eqnarray}
\label{maineq}
\sigma_1(x^*)[u_1(x^*)]_i\frac{\partial[v_1(x^*)]_j}{\partial x_k}+\frac{\partial\sigma_2(x^*)}{\partial x_k}[u_2(x^*)]_i[v_2(x^*)]_j=
\frac{\partial^2 g_i(x^*)}{\partial x_k\partial x_j}.
\end{eqnarray}
For fixed $j=1,2$ and $k=1,2$, we can multiply equation (\ref{maineq}) by $[u_1(x^*)]_i$ and add for $i=1,2,3$ to get:
\begin{eqnarray*}
\sigma_1(x^*)\|u_1(x^*)\|^2\frac{\partial[v_1(x^*)]_j}{\partial x_k}+\frac{\partial\sigma_2(x^*)}{\partial x_k}\langle u_1(x^*),u_2(x^*)\rangle[v_2(x^*)]_j=
\langle u_1(x^*),\left[\frac{\partial^2 g_i(x^*)}{\partial x_k\partial x_j}\right]_{i=1}^3\rangle.
\end{eqnarray*}
Since $\sigma_1(x^*)\neq0$, computing derivatives, 
using the definition of $u_1(x^*)$ and the fact that $u_1(x^*)\perp u_2(x^*)$, we conclude that 
$$\frac{\partial[v_1(x^*)]_j}{\partial x_k}=0, j=1,2, k=1,2.$$
Substituting back in \eqref{maineq} we have, for all $i=1,2,3, j=1,2, k=1,2$:
\begin{eqnarray*}
\frac{\partial\sigma_2(x^*)}{\partial x_k}[u_2(x^*)]_i[v_2(x^*)]_j=
\frac{\partial^2 g_i(x^*)}{\partial x_k\partial x_j}.
\end{eqnarray*}
At indices $(i,j,k)\in\{(1,1,2),(1,2,1)\}$, where the right-hand side is non-zero, we have $\frac{\partial\sigma_2(x^*)}{\partial x_1}\neq0, \frac{\partial\sigma_2(x^*)}{\partial x_2}\neq0, [v_2(x^*)]_1\neq0$ and $[v_2(x^*)]_2\neq0$. At indices $(i,j,k)\in\{(1,1,1),(2,1,1),(3,1,1)\}$, where the right-hand side is zero, we get $u_2(x^*)=0$, which is a contradiction.
\end{example}

To conclude this section we note that our proof suggests that when the rank is constant, 
the Hessian of the Lagrangian does not depend on the Lagrange multiplier. 
In fact, we can prove this without additional assumptions. 
This explains why results under constant rank conditions hold {\it for all} Lagrange multipliers.

\begin{theorem}
\label{indep}
Suppose that $\Lambda(x^{*})\neq \emptyset$. 
If the rank of $J(x)$ is constant around a point
$x^*$, then 
the quadratic form $d^\mathtt{T}\nabla^2_{xx}L(x^*,\lambda,\mu)d$ for $d\in S(x^*)$ 
does not depend on $(\lambda,\mu)\in\Lambda(x^*)$.
\end{theorem}
\begin{proof}
By simplicity, assume $m=0$ and $A(x^*)=\{1,\dots,p\}$. 
By Theorem \ref{ranktheo}, for each $d\in S(x^*)$, 
there exists a smooth curve $\xi(t)$, $t\in(-\varepsilon,\varepsilon)$ with $g(\xi(t))=0$ for all $t$, with $\xi(0)=x^*$ 
and $\xi'(0)=d$. 
Take $\tilde\mu\in \text{Ker}(J(x^*)^\mathtt{T})$ and let us define 
the function $R(t):=\sum_{i=1}^p\tilde\mu_i g_i(\xi(t))$, which is constantly zero for small $t$.
Straightforward calculations show 
that $R''(0)=d^\mathtt{T}\sum_{i=1}^p\tilde\mu_i\nabla^2 g_i(x^*)d+\xi''(0)^\mathtt{T}J(x^*)^\mathtt{T}\tilde\mu=0$. 
Hence, $d^\mathtt{T}\sum_{i=1}^p\tilde\mu_i\nabla^2 g_i(x^*)d=0$.

But $\Lambda(x^*)=(\bar{\mu}+\text{Ker}(J(x^*)^\mathtt{T}))\cap\R^p_+$ for a fixed 
Lagrange multiplier $\bar\mu \in\Lambda(x^*)$. 
Hence  $\mu \in\Lambda(x^*)$ if, and only if, $\mu=\bar\mu+\tilde\mu$, for some $\tilde\mu\in \text{Ker}(J(x^*)^\mathtt{T})$, 
with $\bar \mu+ \tilde \mu \geq0$. 
It follows that $d^\mathtt{T}\nabla^2_{xx}L(x^*,\mu)d=d^\mathtt{T}\nabla^2_{xx}L(x^*,\bar\mu)d$, as we wanted to show.
Observe that $x^{*}$ is not necessarily a local minimizer, we only require $\Lambda(x^{*})\neq \emptyset$.
\end{proof}

Despite our focus on conditions implying WSOC, the above analysis allows us to obtain conclusions about SSOC, 
related with \cite[Theorem 5.1]{baccaritrad}. Recall 
that the {\it generalized strict complementary slackness} (GSCS) condition holds at the feasible point $x^{*}$ if there exists, at most, one index $i_{*} \in A(x^{*})$ such that $\mu_{i_{*}}=0$ whenever
$(\lambda, \mu) \in \Lambda(x^{*})$.

\begin{theorem}\label{theo:main2} Assume that $x^*$ is a local minimizer of (\ref{problem}) 
that fulfills MFCQ. Let $r$ be the rank of $J(x^*)$, and assume that for every $x$ in 
some neighborhood of $x^*$ the rank of $J(x)$ is at most $r+1$. 
Suppose also that Assumption \ref{svd} and GSCS hold at $x^*$. Then, there is a Lagrange multiplier ($\lambda,\mu)$ such that SSOC holds.
\end{theorem}
\begin{proof}
From \cite[Theorem 5.1]{baccaritrad} or \cite[Definition 3.3 and Lemma 3.3]{abadie2}, 
it follows that GSCS implies that $C(x^*)$ is a first-order cone, 
hence, we can still apply Yuan's Lemma and prove the result in the same lines of Theorem \ref{theo}.
\end{proof}

\section{Final remarks}

In order to analyse limit points of a sequence generated by a second-order algorithm, 
one usually relies on WSOC, the stationarity concept based on the critical subspace, 
the lineality space of the cone of critical directions. 
Most conditions guaranteeing WSOC at local minimizers are based on a constant rank assumption on the Jacobian matrix. 
In this paper we developed new tools to deal with the non-constant rank case, by partially solving 
a conjecture formulated in \cite{ams2}. 
Possible future lines of research includes 
investigating the full conjecture using generalized notions of derivative. We believe this can be done since under the rank assumption, the so-called ``crossing'' of singular values is controlled, at least when the non-zero ones are simple, which is the main source of non-continuity of singular vectors. 
Our approach also opens the path to obtaining new second-order results without assuming MFCQ and/or to developing conditions that ensure SSOC, the second-order stationarity concept based on the true critical cone.

\section*{Acknowledgements}

We thank L. Minchenko for a discussion around Theorem 2.6 after a first draft of this paper was released. This research was funded by FAPESP grants 2013/05475-7 
and 2016/02092-8
, by CNPq grants 454798/2015-6, 
303264/2015-2 
and 481992/2013-8, 
and CAPES. 

\bibliographystyle{plain}   
\bibliography{biblio_conj_nino}

\end{document}